\documentclass[11pt, reqno]{amsart}
\usepackage[normalem]{ulem}
\PassOptionsToPackage{table}{xcolor}
\usepackage{stmaryrd}
\expandafter\def\csname opt@stmaryrd.sty\endcsname
{only,shortleftarrow,shortrightarrow}
\usepackage{extpfeil}
\usepackage{amsmath,amssymb,amsthm}
\usepackage{aliascnt}
\usepackage{varioref}
\usepackage{hyperref}
\usepackage{thmtools}
\usepackage[capitalize,nameinlink,noabbrev]{cleveref}
\hypersetup{colorlinks=true,urlcolor=blue,citecolor=blue,linkcolor=blue}
\usepackage{courier}
\usepackage{tikz}
\usepackage{tikz-cd}
\usetikzlibrary{calc,matrix,arrows,decorations.markings}
\usepackage{array}
\usepackage{color}
\usepackage{enumerate}
\usepackage{nicefrac}
\usepackage{listings}
\usepackage{siunitx}
\usepackage{seqsplit}
\usepackage{longtable}
\usepackage{colonequals}
\usepackage{array, multirow}
\makeatletter
\renewcommand{\pod}[1]{\allowbreak\mathchoice
  {\if@display \mkern 6mu\else \mkern 8mu\fi (#1)}
  {\if@display \mkern 6mu\else \mkern 8mu\fi (#1)}
  {\mkern4mu(#1)}
  {\mkern4mu(#1)}
}

\hypersetup{colorlinks=true,urlcolor=blue,citecolor=blue,linkcolor=blue}

\numberwithin{equation}{section}

\newtheorem{theorem}{Theorem}[section]

\newtheorem{lemma}[theorem]{Lemma}
\newtheorem{proposition}[theorem]{Proposition}
\newtheorem{corollary}[theorem]{Corollary}

\theoremstyle{definition}
\newtheorem{definition}[theorem]{Definition}

\newtheorem{remark}[theorem]{Remark}

\newcommand{\C}{\mathbb C}

\newcommand{\HH}{\mathbb H}

\newcommand{\Q}{\mathbb Q}
\newcommand{\R}{\mathbb R}

\newcommand{\Z}{\mathbb Z}
\renewcommand{\P}{\mathbb{P}}
\newcommand{\SBC}{\textnormal{SBC}}
\newcommand{\CM}{\textnormal{CM}}

\newcommand{\sepdeg}[1]{\textnormal{deg}_{\textnormal{sep}}(#1)}
\newcommand{\insepdeg}[1]{\textnormal{deg}_{\textnormal{insep}}(#1)}

\newcommand{\hmod}[1]{\mathrm{h}_\textnormal{mod}(#1)}

\DeclareMathOperator{\End}{End}
\DeclareMathOperator{\Gal}{Gal}

\DeclareMathOperator{\GL}{GL}
\DeclareMathOperator{\SL}{SL}

\DeclareMathOperator{\PSL}{PSL}
\DeclareMathOperator{\PGL}{PGL}

\DeclareMathOperator{\Spec}{Spec}

\DeclareMathOperator{\Sym}{Sym}

\DeclareMathOperator{\Deg}{Deg}

\makeatletter
\let\amsmath@bigm\bigm

\renewcommand{\bigm}[1]{%
  \ifcsname fenced@\string#1\endcsname
    \expandafter\@firstoftwo
  \else
    \expandafter\@secondoftwo
  \fi
  {\expandafter\amsmath@bigm\csname fenced@\string#1\endcsname}%
  {\amsmath@bigm#1}%
}

\newcommand{\DeclareFence}[2]{\@namedef{fenced@\string#1}{#2}}
\makeatother

\DeclareFence{\mid}{|}

\makeatletter
\@namedef{subjclassname@2020}{\textup{2020} Mathematics Subject Classification}
\makeatother

\title{Effective surjectivity of Galois representations of products of elliptic curves over function fields \\[0.5em] {\lowercase{\normalsize with an appendix with \uppercase{B}enjamin \uppercase{B}akker}}}

\begin{document}



\address{Benjamin Bakker 
\begin{itemize}
\item[-]
Department of Mathematics, Statistics and Computer Science, University of Illinois at Chicago, 851 S Morgan St, 322
SEO, Chicago, 60607, IL, USA
\end{itemize}
}
\urladdr{https://benjamin-bakker.github.io}
\email{bakker@gmail.com}

\author{\sc Alina Carmen Cojocaru}
\address{Alina Carmen Cojocaru
\begin{itemize}
\item[-]
Department of Mathematics, Statistics and Computer Science, University of Illinois at Chicago, 851 S Morgan St, 322
SEO, Chicago, 60607, IL, USA
\item[-]
Institute of Mathematics  ``Simion Stoilow'' of the Romanian Academy, 21 Calea Grivitei St, Bucharest, 010702,
Sector 1, Romania
\end{itemize}
} 
\urladdr{https://www.alinacarmencojocaru.com/}
\email{math.cojocaru@gmail.com}

\author{\sc Frederick Saia}
\address{Frederick Saia 
\begin{itemize}
\item[-]
Department of Mathematics, Statistics and Computer Science, University of Illinois at Chicago, 851 S Morgan St, 322
SEO, Chicago, 60607, IL, USA
\end{itemize}
}
\urladdr{https://fsaia.github.io/site/}
\email{freddy.v.saia@gmail.com}

\begin{abstract}
We prove an effective surjectivity result for Galois representations of products of non-isotrivial, non-isogenous elliptic curves over certain function fields of characteristic $0$. This is by way of an isogeny degree bound in this setting, generated from bounds for elliptic curves by Griffon--Pazuki and from the function field analogue of the Frey--Mazur conjecture, by employing techniques originating in work by Serre and Masser--W{\"{u}}stholz in the number field setting. 
\end{abstract}

\maketitle


\section{Introduction}\label{intro_section}

Let $K$ be an arbitrary field with arbitrary characteristic,  for which we fix algebraic and separable closures $\overline{K}$ and  $K^\textnormal{sep}$, and denote by $G_K$ the absolute Galois group $\Gal(K^\textnormal{sep}/K)$.
Let $E$ be an elliptic curve over $K$ and let $N$ be a positive integer. The  group $G_K$ acts naturally on the torsion points of $E$ over  $K^\textnormal{sep}$, acting, in particular, on the subgroup $E[N]$ of points of order dividing $N$. This action is captured in the mod-$N$ Galois representation 
\[ \rho_{E,N} : G_K \rightarrow \GL(E[N]) \simeq \GL_2(\Z/N\Z)  . \]

When $K$ is a number field and  $E$ does not have complex multiplication over $K^\textnormal{sep}$ (shortened as $E$ is non-CM),
by a celebrated result of Serre \cite[Th{\'{e}}or{\`{e}}me 2]{Serre}, there exists a constant 
$c(E/K)$
such that $\rho_{E,\ell}$ is surjective for all primes 
$\ell > c(E/K)$.
Serre also proved an analogous result \cite[Th{\'{e}}or{\`{e}}me 6]{Serre} for a product of two elliptic curves defined over the number field $K$, namely that 
for two non-CM, non-(geometrically)-isogenous elliptic curves $E_1$ and $E_2$ over $K$,
there exists a constant
$c(E_1/K, E_2/K)$
such that
  the image of the mod-$\ell$ Galois representation associated to the product $E_1 \times E_2$
 is as large as possible
  for all primes $\ell > c(E_1/K, E_2/K)$.

When $K$ is the function field of a smooth, projective, and geometrically integral curve over a perfect field of constants $k$,  a theorem of Igusa \cite[Theorem 3]{Igusa}, whose proof precedes that of Serre's theorem,  implies that
 if $E$ satisfies that $j(E) \not\in \overline{k}$,
 then there exists a constant $c(E/K)$ such that 
the image of the representation $\rho_{E,\ell}$ is as large as possible, in the sense that 
$\text{Image}(\rho_{E,\ell}) \cap \SL_2(\Z/\ell \Z) = \SL_2(\Z/\ell \Z)$,
for all primes  $\ell > c(E/K)$. 
We also refer the reader to \cite{BLV} for an alternate proof of this result.

Naturally, one asks for effective versions of the above results.
 Following Serre's work, an effective theorem for the product of an arbitrary number of 
 non-CM, 
 pairwise
 non-(geometrically)-isogenous elliptic curves 
 elliptic curves $E_1, \ldots, E_n$ 
 over a number field $K$ was proven by Masser--W\"{u}stholz \cite[Theorem, Proposition 1]{MW_Galois} as an application of their work on isogeny degree bounds in \cite{MW_isog}. Specifically, Masser--W\"{u}stholz proved that
 there exists a constant
$c(E_1/K, \ldots, E_n/K)$, 
defined explicitly in terms of the Weil heights of the elliptic curves and the degree of the number field,
such that
  the image of the mod-$\ell$ Galois representation associated to the product $E_1 \times \ldots \times E_n$
 is as large as possible
  for all primes $\ell > c(E_1/K, \ldots, E_n/K)$.
Since their work, the subject of effective surjectivity results over number fields has seen ongoing interest, with many researchers working to improve the bounds of Masser--W{\"{u}}stholz or to extend these types of results to other classes of abelian varieties (see, e.g., \cite{Coj, LV, Lom15, LF16, Lom16a, Lom16b, Zyw, MWa24b, MWa24a}). 
In the function field setting,
an  effective surjectivity result for elliptic curves  was proven by Cojocaru--Hall in \cite{CH}. Their theorem, which we recall as \cref{CH_thm},
shows that, 
when $K$ is the function field of a smooth, projective, and geometrically integral curve of genus $g$ over a perfect field of constants $k$, 
for  any elliptic curve $E$ over  $K$ with $j(E) \not\in {\overline{k}}$,  there exists an explicit constant 
$c(g)$ such that 
the image of the representation $\rho_{E,\ell}$ is as large as possible
for all primes  $\ell \geq c(g)$.

Our main result 
is an effective surjectivity result for the product of an arbitrary number of elliptic curves over certain function fields, providing the first extension of both Igusa's theorem and Cojocaru--Hall's theorem to 
abelian varieties
of dimension larger than $1$. It uses the following notation. 

For $C$ a smooth, projective, and geometrically integral curve of genus $g$ over a perfect field $k$,
let $K := k(C)$ denote the function field of $C$. 
For $E_1,  \ldots, E_n$ elliptic curves over $K$,
consider the product abelian variety of dimension $n$ defined by
\[ A := E_1 \times \cdots \times E_n. \]
 For any prime $\ell \neq \text{char}(K)$, denote by
\begin{align*}
\rho_{A,\ell} : G_K &\rightarrow \GL(E_1[\ell] \times \cdots \times E_n[\ell]) \simeq \GL_2(\Z/\ell\Z)^n
\end{align*}
the mod-$\ell$ Galois representation of $A$, that is,  the product of the mod-$\ell$ Galois representations of the factors $E_i$. 
Since the determinants of the projections of $\rho_{A,\ell}$ to each component must agree, each being equal to the $\ell^\textnormal{th}$ cyclotomic character of $K$,  the image of $\rho_{A,\ell}$ is contained in the determinant locus
\[\left\{ (M_1,  \ldots, M_n) \mid \det{M_1}  = \ldots = \det{M_n}\right\} \leq \GL_2(\mathbb{Z}/\ell\mathbb{Z})^n  .\]
Letting $$\Psi_\ell(A) := \textnormal{Image}(\rho_{A,\ell}) \cap \SL_2(\Z/\ell\Z)^n,$$ 
our effective result 
is as follows (see \cref{effective_thm} for a more detailed version):  

\begin{theorem}\label{effective_theorem_introduction_version}
Let $K$ be the function field of a smooth, projective, and geometrically integral curve $C$ of genus $g$ over a field $k$ with an embedding $k \hookrightarrow \C$. 
For $n \geq 2$,
let $E_1, \ldots, E_n$ be pairwise non-isogenous elliptic curves over $K$ with $j(E_i) \not \in \overline{k}$ for each $1 \leq i \leq n$, and let 
$ A := E_1 \times \cdots \times E_n$.
Then there exists a constant $\mathfrak{c}(g)$, depending only on $g$, such that 
$\Psi_\ell(A) = \SL_2(\Z/\ell\Z)^n$
 for all primes $\ell > \mathfrak{c}(g)$. 
\end{theorem}

Similarly to \cref{CH_thm},
a pleasing aspect of \cref{effective_theorem_introduction_version}
is that 
the constant $\mathfrak{c}(g)$
depends only on the genus $g$ of the base field 
and not on the individual elliptic curve factors.
In fact, $\mathfrak{c}(g)$ does not even depend on the dimension of $A$.  

It should be possible to make $\mathfrak{c}(g)$ explicit.
As we clarify in the statement of \cref{effective_thm}  (see also \cref{explicit_remark} for related comments),
 it would suffice to make explicit a result of Bakker--Tsimerman \cite{BT16} which proves a function field analogue of the Frey--Mazur conjecture. 
 Similarly, to obtain from our work a version of \cref{effective_theorem_introduction_version} over function fields with more general fields of constants, it would suffice to have a version of \cite[Theorem 2]{BT16} for such function fields (see \cref{characteristic_remark}). 

Our proof of \cref{effective_theorem_introduction_version} follows the techniques of Masser--W\"{u}stholz from \cite{MW_Galois}, which, in our setting, require as input an isogeny degree estimate for product abelian surfaces over function fields. This input comes in the form of \cref{bisep_bound_genus_char0}, which we prove by bootstrapping isogeny degree estimates for elliptic curves over function fields, obtained in recent work of Griffon--Pazuki \cite[Proposition 6.5, Proposition 6.7]{GP22}. Our argument, like that of Griffon--Pazuki in the dimension $1$ case, involves points on modular curves $X_0(N)$, of specified level $N$, induced by minimal-degree isogenies, along with genus bounds for $X_0(N)$ in terms of the level $N$. As hinted above, the argument also uses a variation of an effective version of the Frey--Mazur conjecture in the function field setting \cite[Corollary 30]{BT16}, proved by Bakker--Tsimerman. Specifically, 
it relies on
a generalization of this result to composite level, handled in \cref{appendix} and coauthored with Bakker. 

Our paper is structured as follows. In \cref{background_section}, we briefly review necessary preliminaries related to function fields, heights, and isogenies.  In \cref{Frey_Mazur_section}, we state our generalization of the function-field Frey--Mazur result of Bakker--Tsimerman, which we justify in \cref{appendix}.  We  prove isogeny degree bounds in \cref{deg_bound_section} and apply one such bound to prove \cref{effective_theorem_introduction_version} in \cref{effective_surjectivity_section}.

\subsection*{Acknowledgments}
We thank Nathan Jones  and Oana Padurariu for feedback on earlier drafts of this work, and we thank Ben Bakker and Ariel Weiss for helpful conversations. We also express our gratitude to the anonymous referees who helped us to identify issues in prior versions of this work, in particular by pointing out errors in a prior version of \cref{bisep_bound_genus}.


\section{Background material}\label{background_section}

\subsection{Function fields and heights}
For the remaining of the paper (with the exception of \cref{appendix} ), unless noted otherwise, $K$ will denote a function field, defined as follows.
Let $C$ be a smooth, projective, and geometrically integral curve of genus $g$ over a perfect field $k$ of characteristic $p := \textnormal{char}(k) \geq 0$. We let $K := K(C)$ denote the function field of $C$ and call it a {\bf{function field}}
with {\bf{field of constants}} $k$
 and 
 of {\bf{genus}} $g$ (noting that $g$ is independent of the choice of model $C$).
While we will later restrict to the case $\text{char}(k) = 0$, we work in general as much as possible to make clear to the reader what would be needed to generalize our main result to nonzero characteristic.
For general background on function fields, we refer the reader to \cite{Rosenff}.  

Let $L/K$ be a finite degree extension. For each place $w$ of $L$, let $|\cdot|_w$ denote the corresponding normalized absolute value on $L$. Let $v$ denote the unique place of $K$ lying below $w$ and let $L_w$ and $K_v$, respectively, denote the corresponding completions of $L$ and $K$. The local degree of $w$ is the quantity
\[ n_w := [L_w : K_v]. \]

For a point $P = [x_0 : x_1] \in \mathbb{P}^1(L)$,  the height of $P$ with respect to $L$ is 
\[ h_L(P) := \sum_{w \textnormal{ a place of } L} n_w \log \max\{|x_0|_w, |x_1|_w\}. \]
The (absolute logarithmic) \textbf{Weil height} of $P \in \mathbb{P}^1(L)$,
which is independent of the choice of extension $L$ such that $P \in \mathbb{P}^1(L)$,
 is defined as
\[ h(P) := \frac{h_L(P)}{[L : K]}. \]
For $f \in \overline{K}$, the \textbf{Weil height} of $f$ is then defined as
\[ h(f) := h([f : 1]). \]

Like results of \cite{MW_Galois} and \cite{GP22}, versions of our effective results on products and isogeny degree bounds will involve the modular heights of elliptic curves. For this purpose, we recall the definition of modular height from \cite[\S 1]{GP22}. 
The \textbf{modular height} of 
an elliptic curve $E$ over $\overline{K}$, with $j$-invariant $j(E)$, 
is the (absolute logarithmic) Weil height of $j(E)$, that is,
\[ \hmod{E} := h(j(E)). \]

\subsection{Isogenies}

For general background on abelian varieties, we recommend \cite{EGM, Mil, Mum}. We will most often reference \cite{EGM}, which from the start works over a general (not necessarily algebraically closed) base field, whenever possible.

Unless otherwise specified, in this paper isogenies will be assumed to be defined over $\overline{K}$ (equivalently, over some finite extension of $K$). 

For abelian varieties $X/K$ and $Y/K$, of the same dimension $\textnormal{dim}(X) = \textnormal{dim}(Y)$, with respective identity elements $0_X$ and $0_Y$, we let 
$\textnormal{Hom}_{\overline{K}}(X,Y)$ denote the ring consisting of all isogenies $\left( X \otimes_{\Spec K} \Spec \overline{K}\right) \rightarrow Y \otimes_{\Spec K} \Spec \overline{K}$ between $X$ and $Y$ over $\overline{K}$, along with the constant map sending every element of $X$ to $0_Y$. 
As a particular case, we let $\textnormal{End}_{\overline{K}}(X)$ denote $\textnormal{Hom}_{\overline{K}}(X,X)$,
and for an arbitrary positive integer $d$, we denote by $[d]_X \in \textnormal{End}_{\overline{K}}(X)$ the multiplication-by-$d$ isogeny.
For an isogeny $\varphi \in \textnormal{Hom}_{\overline{K}}(X,Y)$,
its degree $\textnormal{deg}(\varphi)$ is
the degree $[K(X) : \varphi^*(K(Y))]$ of the corresponding extension of function fields, where $\varphi^* : K(Y) \to K(X)$ denotes the induced pullback map.
Equivalently, $\textnormal{deg}(\varphi)$ is the rank of the finite group scheme $\textnormal{ker}(\varphi)$ \cite[p. 73]{EGM}. 

The following result is due to Deuring in positive characteristic, and its statement and a proof, which is agnostic to the characteristic, can be found as \cite[Lemma 2.1]{GP22}. 

\begin{lemma}
Let $K$ be a function field with field of constants $k$, of arbitrary characteristic, and let $E$ be an elliptic curve over $K$ with $j(E) \not \in \overline{k}$. Then $\textnormal{End}_{\overline{K}}(E) \simeq \mathbb{Z}$. 
\end{lemma}

We call an elliptic curve $E/K$  \textbf{isotrivial} if $j(E) \in \overline{k}$, and  \textbf{non-isotrivial} otherwise. Recalling that an elliptic curve with CM over a number field must have integral $j$-invariant, the above lemma should be interpreted 
 as an analogue over function fields: if $E/K$ is non-isotrivial (i.e., it has non-constant $j$-invariant), then it does not have extra endomorphisms. 

\subsubsection{Factorizations}

The main thrust of this subsection is to associate to an isogeny  an auxiliary isogeny in the other direction; this construction will be relevant to our definition of biseparability to come. Some of the work here is not strictly necessary for the remainder of this paper, but is useful in identifying how one might seek to generalize our isogeny-degree bound in \cref{deg_bound_section}.

\begin{proposition}\label{isog_fact}
Let $K$ be a function field  of arbitrary characteristic, 
let $X, Y$ be abelian varieties over $K$, and let $\varphi \in \textnormal{Hom}_{\overline{K}}(X,Y)$. Let $H \subseteq \textnormal{ker}(\varphi)$ be a subgroup scheme and let $\eta : X \rightarrow X/H$ denote the natural quotient map. Then there exists an isogeny 
$\psi \in \textnormal{Hom}_{\overline{K}}(X/H,Y)$
such that 
$\varphi = \psi \circ \eta$. 
\end{proposition}
\begin{proof}
The main claim here is the implicit one that $X/H$ is again an abelian variety; this follows from \cite[(4.39) Theorem]{EGM}. We then get the isogeny $\psi$ from the universal property for quotients. 
\end{proof}

The following result will be useful when we need certain factorizations of isogenies to be unique.

\begin{lemma}{\cite[(5.4) Lemma]{EGM}}\label{unique_fact_lemma}
Let $K$ be a function field  of arbitrary characteristic, 
let $X, Y, Z, W$ be abelian varieties over $K$,
and let 
$\alpha \in \textnormal{Hom}_{\overline{K}}(W, X)$,
$\beta \in \textnormal{Hom}_{\overline{K}}(Y, Z)$,
$\varphi_1, \varphi_2 \in \textnormal{Hom}_{\overline{K}}(X, Y)$
 be isogenies
such that
$\beta \circ \varphi_1 \circ \alpha = \beta \circ \varphi_2 \circ \alpha$. Then $\varphi_1 = \varphi_2$. 
\end{lemma}

We now introduce for an isogeny $\varphi$ an associated isogeny  $\widetilde{\varphi}$ in the opposite direction.

\begin{proposition}\label{tilde_prop}
Let $K$ be a function field  of arbitrary characteristic, 
let $X, Y$ be abelian varieties over $K$,
and let 
$\varphi \in \textnormal{Hom}_{\overline{K}}(X, Y)$
be an isogeny of degree $d$.
Then there exists a unique isogeny $\widetilde{\varphi} \in \textnormal{Hom}_{\overline{K}}(Y, X)$ such that 
\[ \widetilde{\varphi} \circ \varphi = [d]_X \qquad \textnormal{ and } \qquad \varphi \circ \widetilde{\varphi} = [d]_Y. \]
\end{proposition}
\begin{proof}
This is \cite[(5.12) Proposition]{EGM}. While the uniqueness claim is not explicitly made there,  it follows immediately from \cref{unique_fact_lemma}
\end{proof}

We remark that, when the dimension of $X$ and $Y$ is greater than $1$, the isogeny $\widetilde{\varphi}$ is \emph{not} the dual isogeny of $\varphi$ (note that we have taken care to say nothing about polarizations). In particular, one can see by comparing degrees that the degree of $\varphi$ is smaller than that of $\widetilde{\varphi}$ in dimension greater than $1$, and, as a result, $\widetilde{\widetilde{\varphi}} : X \rightarrow Y$ is not equal to $\varphi$. 

Related to the above remark,  we note the following properties, which can be found in \cite{Rosen} for isogenies over $\mathbb{C}$. While Rosen \emph{does} refer to $\widetilde{\varphi}$ as the ``dual isogeny,''  we would label this a misnomer.

\begin{proposition}
Let $K$ be a function field  of arbitrary characteristic, 
let $X, Y, Z$ be abelian varieties over $K$ of dimension $n$,
and let 
$\varphi \in \textnormal{Hom}_{\overline{K}}(X, Y)$,
$\varphi' \in \textnormal{Hom}_{\overline{K}}(Y, Z)$
be  isogenies of degrees $d$, $d'$, respectively.
Then:
\begin{enumerate}
\item[(i)] $\textnormal{deg}\left(\widetilde{\varphi}\right) = d^{2n-1}$.
\item[(ii)] $\widetilde{\widetilde{\varphi}} = [d^{2n-2}]_Y \circ \varphi$.
\item[(iii)] $\widetilde{\varphi' \circ \varphi} = \widetilde{\varphi} \circ \widetilde{\varphi'}$. 
\item[(iv)] For all positive integers $m$, we have $\widetilde{[m]_X} = [m^{2n-1}]_X$. 
\end{enumerate}
\end{proposition}

\begin{proof}
Part (i) is clear from comparing degrees, given that $\widetilde{\varphi} \circ \varphi = [d]_X$ and $\textnormal{deg}([d]_X) = d^{2n}$. For part (ii), note that 
\begin{align*}
\left( [d^{2n-2}]_Y \circ \varphi \right) \circ \widetilde{\varphi} &= [d^{2n-2}]_Y \circ [d]_Y \\
&= [d^{2n-1}]_Y \\
&= \widetilde{\widetilde{\varphi}} \circ \widetilde{\varphi}. 
\end{align*}
The claim then follows from \cref{unique_fact_lemma}. Part (iii) is seen as follows:
\begin{align*}
\left(\widetilde{\varphi} \circ \widetilde{\varphi'} \right) \circ \left( \varphi' \circ \varphi \right) &= \widetilde{\varphi} \circ [d']_Y \circ \varphi \\
&= [dd']_X \\
&= \widetilde{\varphi' \circ \varphi} \circ \left( \varphi' \circ \varphi \right). 
\end{align*}
Part (iv) is now a routine verification.
\end{proof}

\subsubsection{Separability and biseparability}
We are now ready to introduce our definition of biseparability.
For this,  
let $X, Y$ be abelian varieties over the function field $K$  of arbitrary characteristic
and let 
$\varphi \in \textnormal{Hom}_{\overline{K}}(X, Y)$.
We set
\[  \sepdeg{\varphi} := [\overline{K}(X) : \varphi^*(\overline{K}(Y))]_{\textnormal{sep}}\]
and
\[ \insepdeg{\varphi} := [\overline{K}(X) : \varphi^*(\overline{K}(Y))]_{\textnormal{insep}} \]
to be the separability and inseparability degrees, respectively, of the extension of function fields 
$\overline{K}(X)/\varphi^*(\overline{K}(Y))$.
Equivalently, $\sepdeg{\varphi}$ is equal to the number of closed points of $\textnormal{ker}(\varphi)$ and 
$\sepdeg{\varphi} \cdot \insepdeg{\varphi} = \textnormal{deg}(\varphi)$.
We call $\varphi$ \textbf{separable} if the field extension
$\overline{K}(X)/\varphi^*(\overline{K}(Y))$ 
is separable, i.e., if $\insepdeg{\varphi} = 1$, and \textbf{purely inseparable} if this field extension is purely inseparable, i.e., if $\sepdeg{\varphi} = 1$. We call $\varphi$ \textbf{biseparable} if both $\varphi$ and $\widetilde{\varphi}$ are separable. 

Note that, if $\textnormal{char}(K) = 0$, then all isogenies of abelian varieties over $K$ are separable and hence also biseparable. This does not necessarily hold if $\textnormal{char}(K) > 0$.

\begin{proposition}\cite[Proposition, p.64]{Mum}\label{mult_by_n_prop}
Let $K$ be a function field  of positive characteristic $p = \textnormal{char}(K) > 0$,
let $X$ be an abelian variety over $K$ of dimension $\textnormal{dim}(X)$, and let $d$ be a positive integer.
Then
the multiplication-by-$d$ isogeny $[d]_X : X \rightarrow X$  is of degree $d^{2\cdot \textnormal{dim}(X)}$, and is separable if and only if $d$ is coprime to $p$.
\end{proposition}

\begin{corollary}\label{bisep_if_deg_coprime}
Let $K$ be a function field  of positive characteristic $p = \textnormal{char}(K) > 0$,
let $X, Y$ be abelian varieties over $K$,
and let 
$\varphi \in \textnormal{Hom}_{\overline{K}}(X, Y)$ be an isogeny of degree $d$.
Then  $\varphi$  is biseparable if and only if $d$ is coprime to $p$. 
\end{corollary}
\begin{proof}
This follows from \cref{mult_by_n_prop} and the fact that $\widetilde{\varphi} \circ \varphi = [d]_X$. 
\end{proof}

\begin{remark}
Our definition of biseparable is a generalization of that from \cite{GP22}. In an effort to be as general as possible and not require polarizations, we believe that working with $\widetilde{\varphi}$ rather than the dual isogeny $\varphi^\vee$ of $\varphi$ is the right choice in this paper. 
\end{remark}

Now, let us show that if $p = \textnormal{char}(K) > 0$, then separability does not imply biseparability. As a primary example, for $X/K$ an abelian variety, consider the \textbf{Frobenius homomorphism} 
\[ F_X : X \rightarrow X^{(p)}, \]
which is purely inseparable, of degree $p^{\textnormal{dim}(X)}$ \cite[Proposition 5.15]{EGM},
and  the \textbf{Verschiebung homomorphism}
\[ V_X : X^{(p)} \rightarrow X, \]
which is also 
of degree $p^{\textnormal{dim}(X)}$. We have that $V_X \circ F_X = [p]_X$, $F_X \circ V_X = [p]_{X^{(p)}}$, and that these homomorphisms are Cartier dual to one-another \cite[Proposition 5.19, Proposition 5.20]{EGM}. The separability degree of $V_X$ is intimately related to the structure of the $p$-power torsion of $X$.

\begin{proposition}\label{p-rank_prop}
Let $K$ be a function field  of positive characteristic $p = \textnormal{char}(K) > 0$,
and let $X$ be an abelian variety over $K$.
Then there exists an integer $r = r(X)$ with $0 \leq r \leq \textnormal{dim}(X)$  such that 
\[ X[p^m] \simeq \left(\mathbb{Z}/p^m\mathbb{Z}\right)^r \]
for all positive integers $m$. Explicitly, we have that $$p^r = \sepdeg{V_X}.$$ 
Moreover, the integer $r$ is invariant under isogeny, that is, if $Y$ is another abelian variety over $K$ such that $X$ is isogenous to $Y$, then $r(X) = r(Y)$. 
\end{proposition}
\begin{proof}
This is \cite[Proposition 5.22]{EGM}. The fact that $p^r = \sepdeg{V_X}$ is not explicitly stated in their proposition, but is inherent in their proof. 
\end{proof}

It is a result of Deuring (see \cite[p. 217]{Mum}) that, in the setting of \cref{p-rank_prop} with $X/K$  an abelian variety of dimension $1$, that is, an elliptic curve which we relabel $E/K$, non-isotriviality is equivalent to $r(E) = 1$. Both properties are therefore equivalent to $V_E$ being separable. In this context, $F_E$ and $V_E$ are dual isogenies, and $V_E$ is a separable isogeny which is not biseparable. 

Remaining in the setting of a function field $K$ with
$p = \textnormal{char}(K) > 0$,
let $E_1, \ldots, E_n$ be non-isotrivial, pairwise non-isogeneous elliptic curves over $K$
and 
let $A$ be an abelian variety over $K$,  isogenous to the product $E_1 \times \ldots \times E_n$.
Since for any positive integer $m$, we have
\[ \left(E_1 \times \cdots \times E_n \right)[p^m] \simeq E_1[p^m] \times \cdots \times E_n[p^m] \simeq \left( \mathbb{Z}/p^m\mathbb{Z}\right)^n,\]
we obtain from \cref{p-rank_prop}
 that $r(A) = n = \textnormal{dim}(A)$. 
 Therefore, once again  the Verschiebung  $V_A$ is a separable isogeny which is  not biseparable. This will be the context of the upcoming results of this paper. 

Note that, in the above context, $A$ cannot be defined over a finite field. If there were some finite extension of the field of constants $k$ of $K$ over which we had a model for $A$, then there would exist some further finite extension over which an isogeny $\varphi: A \rightarrow E_1 \times \cdots \times E_n$ is defined. The product $E_1 \times \cdots \times E_n$ must then be defined over said extension, and upon a further finite extension, each $E_i$ would be defined, contradicting our non-isotriviality assumption. 

\begin{remark}
The preceding paragraph addresses the case in which $A$ is isogenous to a totally decomposable abelian variety. For simple abelian varieties over a function field $K$ of positive characteristic, a result due to Tate and Grothendieck states that a simple abelian variety over $K$ has a model over a finite field if and only if it is of CM type \cite[p. 220]{Mum}. For elliptic curves, ``CM type'' is equivalent to isotriviality. 
\end{remark}

Even though the following lemma will not be explicitly used in what follows,  we include it here as it is relevant to the above discussion and we expect that it may be of use in generalizing results of \cref{deg_bound_section} to dimension $d > 2$ in positive characteristic. 

\begin{lemma}\label{kernel_separable}
Let $K$ be a function field  of  positive characteristic $p = \textnormal{char}(K) > 0$,
let $X, Y$ be abelian varieties over a finite extension $L$ of $K$, 
and let
$\varphi \in \textnormal{Hom}_{\overline{K}}(X, Y)$ be a biseparable isogeny.
Assume that  $X$ and $Y$  have no abelian subvarieties of CM type. Then, letting $H := \textnormal{ker}(\varphi)(\overline{K})$, we have that $L(H)/L$ is a separable field extension. 
\end{lemma}
\begin{proof}
This proof follows similarly to that in \cite[Lemma 4.6]{GP22}: 
setting $d = \textnormal{deg}(\varphi)$,
we know that $H \subseteq X[d]$, and because $d$ is coprime to $p$, it follows from \cite[\S 1]{ST68} that $L(X[d])/L$ is separable. Hence, $L(H)/L$ is separable. 
\end{proof}


\section{Congruences of elliptic curves over function fields}\label{Frey_Mazur_section}

In this section, we recall a bound due to Bakker--Tsimerman related to congruences of elliptic curves over function fields, which we will use in the following section. First, we recall the notion of an $N$-congruence of elliptic curves. 

Let $N$ be a positive integer and let $E_1, E_2$ be elliptic curves over an arbitrary field $F$ of characteristic not dividing $N$. An \textbf{$N$-congruence} between $E_1$ and $E_2$ is an isomorphism of $G_F$-modules between $E_1[N]$ and $E_2[N]$. 

The Frey--Mazur conjecture states that there exists a constant $N_0$ such that there exists no $N$-congruence between elliptic curves over $\Q$ for all integers $N > N_0$. We refer the reader to \cite[\S 1]{Fisher14}, \cite[\S 1.2]{FK22}, and \cite[\S 1]{Frengley24} as useful sources for background on this conjecture and related work. For the natural analogue of the Frey--Mazur conjecture in the function field setting, we have the following effective result of Bakker--Tsimerman.

\begin{theorem}\cite[Theorem 2]{BT16}\label{BT_thm}
Let $K$ be the function field of a smooth, projective, and irreducible\footnote{Given that we are working over an algebraically closed field, a curve being irreducible is equivalent to being geometrically integral.}
 curve 
of genus $g$ over $\mathbb{C}$. Let $E_1, E_2$ be non-isotrivial, non-isogenous elliptic curves over $K$. There exists a constant $\mathcal{N}(g)$, depending only on $g$, such that if there exists an $N$-congruence between $E_1$ and $E_2$, then $N \leq \mathcal{N}(g)$. 
\end{theorem}

\begin{remark}\label{BT_composite_remark}
The exact statement of \cref{BT_thm} in \cite{BT16} only gives a bound which excludes 
$N$-congruences 
of elliptic curves over $K$ for sufficiently large primes
 $N$, 
though their methods can be adapted to give the result for general positive integers $N$. While this is not particularly surprising, it is perhaps best to give some details on this extension, which we do in \cref{appendix}. The main result in \cite{BT16}
also is phrased in terms of the gonality of $K$, rather than the genus, but again their methods are suitable to give a bound as in \cref{BT_thm} (this is weaker, and immediately follows from their statement \cite[Corollary 30]{BT16}). We are satisfied with this dependency, as other results we will use 
have 
a similar 
dependence on the genus.
\end{remark}

\begin{remark}
Because the degrees of cyclic isogenies of non-isotrivial elliptic curves over a function field $K$ are uniformly bounded in terms of the genus of $K$ (see \cref{GP_genus}), the ``non-isogenous'' hypothesis in \cref{BT_thm} is not necessary (though the required constant $\mathcal{N}(g)$ may increase upon dropping this hypothesis). We prefer the statement above for this paper, as we will always work with non-isogenous curves.
\end{remark}


\section{An isogeny degree bound}\label{deg_bound_section}
As in Section \ref{background_section}, 
let $C$ be a smooth, projective, and geometrically integral curve of genus $g$ over a perfect field $k$ of characteristic $p := \textnormal{char}(k) \geq 0$, and consider the function field  $K = K(C)$. 

The following statement, providing an upper bound on the order of a $G_K$-stable subgroup of an elliptic curve over $K$ in terms of just the genus of $K$, is \cite[Proposition 6.5]{GP22}. 
It is deduced 
from a clever observation that, in the function field setting, a $K$-rational point on the modular curve $X_0(d)$ induces a morphism $C \to X_0(d)$ from our smooth model $C$ of $K$ to $X_0(d)$ over the constant field $k$.
\begin{lemma}\label{genus_bound}
Let $K$ be a function field  of arbitrary characteristic and genus $g$,
let $E$ be a non-isotrivial elliptic curve over $K$,
and let $H \subseteq E$ be a cyclic, $G_K$-stable subgroup of order $d$. 
If $d$ is coprime to $\textnormal{char}(K)$, then 
\[ d \leq 49 \cdot \textnormal{max}\{1,g\}. \]
\end{lemma}

In their work, Griffon--Pazuki use \cref{genus_bound} to obtain the following degree bounds for biseparable isogenies of elliptic curves over function fields. 

\begin{proposition}\cite[Proposition 6.7]{GP22}\label{GP_genus}
Let $K$ be a function field  of arbitrary characteristic and genus $g$,
and
let $E_1, E_2$ be two non-isotrivial elliptic curves over $K$. 
Assume that there exists a biseparable isogeny 
$\varphi \in \textnormal{Hom}_{\overline{K}}(E_1, E_2)$ defined over $K$.
Then there exists a biseparable isogeny 
$\varphi_0 \in \textnormal{Hom}_{\overline{K}}(E_1, E_2)$ defined over $K$
with 
\[ \deg{(\varphi_0)} \leq 49 \cdot \max\{1,g\}. \]
\end{proposition}

\begin{proposition}\label{GP_hmod}
Let $K$ be a function field  of arbitrary characteristic,
 let $L/K$ be a finite extension,
 and
let $E_1, E_2$ be two non-isotrivial elliptic curves over $L$.
 Assume that there exists a biseparable isogeny 
 $\varphi \in \textnormal{Hom}_{\overline{K}}(E_1, E_2)$ defined over $L$.
 Then there exists a biseparable isogeny 
 $\varphi_0 \in \textnormal{Hom}_{\overline{K}}(E_1, E_2)$ defined over $L$
 with 
\[ \deg{(\varphi_0)} \leq [L : K] \cdot \min\{\hmod{E_1},\hmod{E_2}\} . \]
\end{proposition}
\begin{proof}
This follows from the same argument given for \cite[Proposition 6.7]{GP22}, but using the cruder degree bound of \cite[Proposition 6.6]{GP22} in place of \cite[Proposition 6.5]{GP22}. 
\end{proof}

In the following proposition and its corollary, we leverage \cref{GP_genus} and \cref{BT_thm} to obtain analogous results for a product of two elliptic curves. 

\begin{proposition}\label{bisep_bound_genus}
Let $K$ be a function field  of arbitrary characteristic and genus $g$,
let $E_1, E_2$ be non-isotrivial, non-isogenous elliptic curves over $K$,  with identity elements denoted  $0_1, 0_2$, respectively,
and let $X$ be an abelian surface over $K$. 
Assume that there exists a biseparable isogeny 
$\varphi \in \textnormal{Hom}_{\overline{K}}(E_1 \times E_2, X)$ defined over $K$ 
and let $\varphi_0$ be such an isogeny of minimal degree.
 Setting
\[
E_1' := \varphi_0\left(E_1 \times \{0_2\}\right), \quad 
E_2' := \varphi_0\left(\{0_1\} \times E_2\right),
\]
we have
\[ \textnormal{deg}(\varphi_0) \leq 49^2 \cdot  |(E_1' \cap E_2')(\overline{K})| \cdot \max\{1, g^2\}. \]
\end{proposition}
\begin{proof}
The start of our proof will be similar to that of \cite[Proposition 6.7]{GP22}, with the added task of relating our isogeny degree bound to induced points on modular curves. We begin by clarifying our setup.

Set
\[ d_0 := \textnormal{min}\{d \in \mathbb{Z}^+ \mid \exists \textnormal{ a biseparable isogeny } E_1 \times E_2 \rightarrow X \text{ of degree } d \text{ over } K \}. \]
This integer is well-defined thanks to our hypothesis that there exists a biseparable isogeny over $K$ between the abelian surfaces $E_1 \times E_2$ and $X$. 
In fact, $d_0$ is the degree of the biseparable isogeny of minimal degree,
$\varphi_0 : E_1 \times E_2 \rightarrow X$. 
By \cref{bisep_if_deg_coprime}, we must have $p \nmid d_0$ if $p = \textnormal{char}(K) > 0$.

Each of the elliptic curves $E_1', E_2'$
is an abelian subvariety of $X$.
Considering the isogenies
$\varphi_1 : E_1 \longrightarrow E_1'$,
$P_1 \longmapsto \varphi_0(P_1, 0_2)$, 
 and
$\varphi_2 : E_2 \longrightarrow E_2'$, 
$P_2 \longmapsto \varphi_0(0_1, P_2)$,
we see that
$\textnormal{ker}(\varphi_1)(\overline{K}) \times \{0_2\} \leq \textnormal{ker}(\varphi_0)(\overline{K})$ and $\{0_1\} \times \textnormal{ker}(\varphi_2)(\overline{K}) \leq \textnormal{ker}(\varphi_0)(\overline{K})$. 
Moreover, if $(P_1, P_2) \in \textnormal{ker}(\varphi_0)(\overline{K})$, 
then, with $0_X$ denoting the identity of $X$, 
we find that
$\varphi_1(P_1) + \varphi_2(P_2) = \varphi_0(P_1, P_2) = 0_X$, 
or equivalently $\varphi_1(P_1) = \varphi_2(-P_2)$. Consequently, we have $\varphi_1(P_1) \in (E_1' \cap E_2')(\overline{K})$.


For a given point $R \in E_1' \cap E_2'$ and for each $1 \leq i \leq 2$, there are $d_i := \deg(\phi_i)$ points $P_i \in E_i$ with $\varphi_i(P_i) = R$.
Moreover, each pair of points $(P_1, P_2) \in E_1(\overline{K}) \times E_2(\overline{K})$ satisfying $\varphi_1(P_1) = \varphi_2(P_2) = R$ 
also satisfies
$(P_1, P_2) \in \textnormal{ker}(\varphi_0)(\overline{K})$. 

The fact that $E_1$ and $E_2$ are non-isogenous implies that $|(E_1' \cap E_2')(\overline{K})|$ is finite 
(as $E_1$ and $E_2$ are irreducible curves, if they intersect at infinitely many points then they must in fact be identical).
Then, in light of the above observations, we obtain that
\begin{equation}\label{degree_eqn}
d_0 = |\textnormal{ker}(\varphi_0)(\overline{K})| = d_1 d_2 \cdot |(E_1' \cap E_2')(\overline{K})|.
\end{equation}
Recalling that if $p = \textnormal{char}(K) > 0$, we have $p \nmid d_0$,   we also have
$p \nmid d_i$ for each $1 \leq i \leq 2$.

We claim that both $\varphi_1$ and $\varphi_2$ are cyclic isogenies of elliptic curves. Suppose, for example, that $\varphi_1$ is not cyclic, so there exists some integer $m > 1$ such that 
$m \mid d_1$ and $\varphi_1$ factors through the endomorphism $[m]_{E_1} : E_1 \to E_1$. It follows that $\varphi_0$ factors through the endomorphism 
\[ \eta :=  [m]_{E_1} \times \textnormal{Id}_{E_2} \in \End_{\overline{\Q}}(E_1 \times E_2) \simeq \Z^2. 
\]
Then, by \cref{isog_fact}, there exists an isogeny $\psi : E_1 \times E_2 \rightarrow X$ with $\varphi_0 = \psi \circ \eta$. By comparing degrees, we find that $\psi$ is both biseparable and of degree strictly less than $d_0$, contradicting the minimality of the degree of $\varphi_0$.

Now, by the same arguments as in the proof of \cite[Proposition 6.7]{GP22} and the fact that $d_1, d_2 \mid d_0$, both $\varphi_1$ and $\varphi_2$ are cyclic and biseparable with kernel stable under the action of $G_K$. Thus,
for each $1 \leq i \leq 2$,
 $\varphi_i$ induces a $K$-rational point on $X_0(d_i)$. \cref{genus_bound} then provides the inequalities
\[ d_i \leq 49 \cdot \textnormal{max}\{1, g\}, \]
which, used in (\ref{degree_eqn}), complete the proof.
\end{proof}

\begin{corollary}\label{bisep_bound_genus_char0}
Let $K$ be a function field  of of a smooth, projective, irreducible curve $C$ of genus $g$ over a field $k$ with an embedding $k \hookrightarrow \C$,
let $E_1, E_2$ be non-isotrivial, non-isogenous elliptic curves over $K$,  
and let $X$ be an abelian surface over $K$. 
Assume that there exists a biseparable isogeny 
$\varphi \in \textnormal{Hom}_{\overline{K}}(E_1 \times E_2, X)$ defined over $K$.
Then there exists an isogeny
$\varphi_0 \in \textnormal{Hom}_{\overline{K}}(E_1 \times E_2, X)$ defined over $K$ 
with
\[ \textnormal{deg}(\varphi_0) \leq 49^3 \cdot \mathcal{N}(g)^2 \cdot \textnormal{max}\{1, g^{3}\}, \]
where
$\mathcal{N}(g)$
is the constant introduced in \cref{BT_thm}.
\end{corollary}
\begin{proof}
From \cref{bisep_bound_genus}, it remains to bound 
$|(E_1' \cap E_2')(\overline{K})|$, with 
$E_1'$ and $E_2'$ as defined therein.
 
 First, note that the intersection $(E_1' \cap E_2')(\overline{K})$ is stable under the action of $G_K$: if $\sigma \in G_K$ and $R$ is an element of this intersection, then we can take points $P_1 \in E_1(\overline{K})$ and $P_2 \in E_2(\overline{K})$ with 
 $\varphi_1(P_1) = \varphi_2(P_2) = R$,  obtaining that
\[ R^\sigma = \varphi_1(P_1^\sigma) = \varphi_2(P_2^\sigma). \]

Since this intersection is finite, it is contained in the torsion subgroup of the base change of each $E_i$ to some finite extension of $K$. Thus, there exist positive integers $M \mid N$ such that 
\[ (E_1' \cap E_2')(\overline{K}) \simeq \Z/M\Z \times \Z/N\Z \]
as groups. This isomorphism yields two useful pieces of information:
\begin{itemize}
\item there exists an $M$-congruence between $E_1'$ and $E_2'$ over $K$, and
\item each $E_i'$ has a cyclic $N/M$-isogeny defined over $K$.\footnote{This is because the group $(E_1' \cap E_2')(\overline{K})$, being isomorphic to $\Z/M\Z \times \Z/N\Z$, has a unique cyclic order $N/M$ subgroup which is contained in all of its cyclic order $N$ subgroups. Namely, this is the intersection of all of the cyclic order $N$ subgroups. This cyclic order $N/M$ subgroup must then also be $G_K$-stable.}
\end{itemize}
From the first item, we know by \cref{BT_thm} that $M \leq \mathcal{N}(g)$ (this is where it is needed that $k$ embeds into $\C$). The second item implies by \cref{GP_genus} that
\[ N/M \leq 49 \cdot \max\{1, g\}. \]
Putting this together, we obtain that
\[ |(E_1' \cap E_2')(\overline{K})| = M^2 \cdot (N/M) \leq  49 \cdot \mathcal{N}(g)^2 \cdot \max\{1, g\}, \]
and so from \cref{bisep_bound_genus} we deduce the stated result. 
\end{proof}

\begin{remark}\label{characteristic_remark}
We separate \cref{bisep_bound_genus} and \cref{bisep_bound_genus_char0} into two statements in part to note that the assumption that $k$ embeds into $\C$ in the latter is only required for the application of \cref{BT_thm}. If one had a version of the bound in \cref{BT_thm} over a given field of constants $k$ of arbitrary characteristic, then one could generalize \cref{bisep_bound_genus_char0} accordingly (upon restriction to biseparable isogenies). We hope that this separation also makes clear how one could get a smaller non-uniform bound for certain fixed $E_1, E_2,$ and $X$ out of \cref{bisep_bound_genus} if one had  a biseparable isogeny $E_1 \times E_2 \to X$ with small intersection $(E_1' \cap E_2')(\overline{K})$. 
\end{remark}

\begin{remark}\label{explicit_remark}
We write the bound in \cref{bisep_bound_genus_char0} without absorbing all factors into a single constant to make it clear that the only non-explicit component of this bound is the constant $\mathcal{N}(g)$ coming from \cref{BT_thm}. This is true for all of our results in this paper depending on $\mathcal{N}(g)$: all bounds can be made explicit upon making the bound $\mathcal{N}(g)$ explicit. One of the authors of \cite{BT16} mentioned to us that it should be possible to make their bound explicit following their methods. 
We do not pursue this task in this paper.
\end{remark}

We also provide a version of \cref{bisep_bound_genus} in terms of modular heights, similar to \cref{GP_hmod}.

\begin{proposition}\label{bisep_bound_hmod}
Let $K$ be a function field  of arbitrary characteristic,
let $L$ be a finite extension of $K$,
let $E_1, E_2$ be non-isotrivial, non-isogenous elliptic curves over $L$,  
and let $X$ be an abelian surface over $L$. 
Assume that there exists a biseparable isogeny 
$\varphi \in \textnormal{Hom}_{\overline{K}}(E_1 \times E_2, X)$ defined over $L$.
Then there exists a biseparable isogeny 
$\varphi_0 \in \textnormal{Hom}_{\overline{K}}(E_1 \times E_2, X)$ defined over $L$
with 
\[ \textnormal{deg}(\varphi_0) \leq [L : K]^2 \cdot \hmod{E_1} \cdot \hmod{E_2} \cdot |(E_1' \cap E_2')(\overline{K})|, \]
where
$E_1' := \varphi_0\left(E_1 \times \{0_2\}\right)$,
$E_2' := \varphi_0\left(\{0_1\} \times E_2\right)$,
with  $0_1, 0_2$ denoting the identity elements of $E_1, E_2$, respectively.
\end{proposition}
\begin{proof}
One need only apply \cite[Proposition 6.6]{GP22} in place of \cref{genus_bound} in the proof of \cref{bisep_bound_genus}, as was done to obtain \cref{GP_hmod} compared to \cref{GP_genus} in the elliptic curve case. 
\end{proof}

While \cref{GP_genus} and \cref{bisep_bound_genus_char0} have the pleasing feature of uniformity
in the elliptic curves $E_1, E_2$ and abelian surface $X$, 
depending only on the genus of $K$, 
their non-uniform counterparts \cref{GP_hmod} and \cref{bisep_bound_hmod} 
are also valuable.
For instance, Griffon--Pazuki use \cref{GP_hmod} to prove a bound on sizes of isogeny classes over fixed degrees
\cite[Proposition 6.12]{GP22}; for their purpose, \cref{GP_genus} is too crude,  as genera may be unbounded over a family of extensions of $K$ of a fixed degree.


\section{Effective surjectivity for products}\label{effective_surjectivity_section}

In this section, we apply the results of \cref{deg_bound_section} to obtain an effective open image-type result for Galois representations of products of elliptic curves over a function field. 
An effective result in the case of a single elliptic curve was obtained by Cojocaru--Hall \cite[Theorem 1.1]{CH}, and we will make use of their theorem here. To state it, we recall the following setup.

Let $C$ be a smooth, projective, and geometrically integral curve of genus $g$ over a perfect field $k$ of characteristic $p := \textnormal{char}(k) \geq 0$, and consider the function field  $K = K(C)$. 
For $E$ an elliptic curve over $K$ and $\ell$ a rational prime not equal to $\textnormal{char}(K)$, let
\[ \rho_{E,\ell} : G_K \rightarrow \GL (E[\ell]) \simeq \GL_2(\mathbb{Z}/\ell\mathbb{Z}) \]
denote the mod-$\ell$ Galois representation of $E$, obtained via the action of $G_K$ on the $\ell$-torsion points of $E$. (The isomorphism with $\GL_2(\mathbb{Z}/\ell\mathbb{Z})$ here is non-canonical, depending on a choice of basis for $E[\ell]$.) Letting $$\chi_\ell: G_K \to (\mathbb{Z}/\ell\mathbb{Z})^\times$$
 denote the $\ell^\textnormal{th}$ cyclotomic character of $K$, we have that $\textnormal{det} \circ \rho_{E,\ell} = \chi_\ell$. Hence, the image 
$\textnormal{Image}(\rho_{E,\ell})$
 must be a subgroup of the group $\Gamma_\ell \leq \GL_2(\mathbb{Z}/\ell\mathbb{Z})$ determined by the short exact sequence 
\[ 0 \rightarrow \SL_2(\mathbb{Z}/\ell\mathbb{Z}) \rightarrow \Gamma_\ell \rightarrow \chi_\ell(G_K) \rightarrow 0. \]
That is, $\Gamma_\ell$ is the inverse image in $\GL_2(\mathbb{Z}/\ell\mathbb{Z})$ of the determinant map. Following \cite[\S 1]{CH}, we call 
\[ \Psi_\ell(E) := \textnormal{Image}(\rho_{E,\ell}) \cap \SL_2(\Z/\ell\Z) \subseteq \Gamma_\ell \cap \SL_2(\Z/\ell\Z)= \SL_2(\Z/\ell\Z) \]
the \textbf{geometric Galois group} of $K(E[\ell])/K$. The fixed field of $\Psi_\ell(E)$ corresponds to the constant extension $(K(E[\ell]) \cap \overline{k})K/K$ obtained by adjoining a primitive $\ell^\textnormal{th}$ root of unity, and is invariant under a change in $K$ by a constant extension in this setup (hence the adjective ``geometric'' here). In other words, we have 
$\textnormal{Image}(\rho_{E,\ell}) = \Psi_\ell(E)$ 
upon a base change of $C$ from $k$ to a cyclotomic extension of $k$, and so we restrict our attention to $\Psi_\ell(E)$ for a geometric result. Specifically, the effective result of Cojocaru--Hall states that $\Psi_\ell(E)$ is all of $\SL_2(\Z/\ell\Z)$ for sufficiently large primes $\ell$, where ``sufficiently large'' depends only on the genus of $K$.

\begin{theorem}\label{CH_thm}\cite[Theorem 1.1]{CH}
Let $K$ be a function field  of arbitrary characteristic and genus $g$.
Define
\[ c(g) := 2 + \max\left \{ q \textnormal{ prime} \bigm \mid \dfrac{q - (6+3e_2(q)+4e_3(q))}{12} \leq g \right\}, \]
where, for an arbitrary prime $q$, 
\[
e_2(q) := \begin{cases} 1 \quad &\text{ if } q \equiv 1 \pmod{4}, \\ -1 \quad &\text{ otherwise,} \end{cases} \quad  
\text{and}
\quad e_3(q) := \begin{cases} 1 \quad &\text{ if } q \equiv 1 \pmod{3}, \\ -1 \quad &\text{ otherwise.} \end{cases} \quad
\]
Let $E$ be a  non-isotrivial elliptic curve over $K$.
Then 
$\Psi_\ell(E) = \SL_2(\Z/\ell\Z)$
 for all primes $\ell \geq c(g)$ with $\ell \neq \textnormal{char}(K)$.  
\end{theorem}

To get an effective surjectivity bound for products of such elliptic curves over $K$, we follow the techniques of Masser--W\"{u}stholz \cite{MW_Galois}. In the referenced work, the authors use an upgraded version \cite[Lemma 2.2]{MW_Galois} of an isogeny degree bound  from their prior work \cite{MW_isog} both to prove a result for one single elliptic curve over a number field \cite[Theorem]{MW_Galois}, analogous to \cref{CH_thm}, and to bootstrap this version to products of an arbitrary number of elliptic curves over a number field  \cite[Proposition 1]{MW_Galois}. We will proceed similarly, utilizing instead our isogeny degree bound in the function field setting \cref{bisep_bound_genus_char0}.

The following algebraic result of Masser--W\"{u}stholz, which is a slight generalization of a result of Serre \cite[Lemme 8]{Serre}, will be useful in the course of our proof.

\begin{lemma}\label{MW_lemma}\cite[Lemma 5.1]{MW_Galois}
Let $\ell \geq 5$ be a prime, let $e$ be a divisor of $\ell-1$, let $V_1, V_2$ be vector spaces of dimension $2$ over $\mathbb{F}_\ell$, and let $B_1, B_2$ be the subgroups of $\GL(V_1)$, $\GL(V_2)$, respectively, consisting of all elements whose determinants are $e^\textnormal{th}$ powers. Let $D$ be the subgroup of $B_1 \times B_2$ consisting of all pairs $(b_1, b_2)$ with $\det{b_1} = \det{b_2}$, and suppose $H$ is a subgroup of $D$ in $B_1 \times B_2$ whose projections to each factor are surjective. If $H \neq D$, then there is an isomorphism $f: V_1 \longrightarrow V_2$ and a character $\chi$ of $H$ with $\chi^2 = 1$ such that 
\[ b_2 = \chi(h) f b_1 f^{-1} \]
for all $h = (b_1, b_2) \in H$. 
\end{lemma}

We are ready to state  and prove our main theorem.

\begin{theorem}\label{effective_thm}
Let $K = K(C)$  be the function field of a smooth, projective, and geometrically integral curve $C$ of genus $g$ over a field $k$ with an embedding $k \hookrightarrow \C$. 
Define
\[ \mathfrak{c}(g) = 40353607 \cdot \mathcal{N}(g)^{3} \cdot  \max\{1,g^{9/2}\}, \]
where
$\mathcal{N}(g)$
is the constant introduced in \cref{BT_thm}.
Let $n \geq 2$,
let $E_1, \ldots, E_n$ be non-isotrivial, pairwise non-isogenous elliptic curves over $K$, and let
$A := E_1 \times \cdots \times E_n$. 
For an arbitrary prime $\ell$, denote by 
\begin{align*} 
\rho_{A,\ell} : G_K &\longrightarrow \GL(E_1[\ell]) \times \cdots \times \GL(E_n[\ell]) \simeq \GL_2(\mathbb{Z}/\ell\mathbb{Z})^n \\
\sigma &\longmapsto (\rho_{E_1,\ell}(\sigma),\ldots, \rho_{E_n,\ell}(\sigma)), 
\end{align*}
the mod-$\ell$ Galois representation of $A$, 
and set
$\Psi_\ell(A) := \textnormal{Image}(\rho_{A,\ell}) \cap \SL_2(\Z/\ell\Z)^n$.
Then
$\Psi_\ell(A) = \SL_2(\Z/\ell\Z)^n$
for all primes $\ell > \mathfrak{c}(g)$. 
\end{theorem}
\begin{proof}
We follow the proof structure of \cite[Proposition 1]{MW_Galois}. 
We circumvent the dependence on 
the number field setting
in the proof of this referenced result
by utilizing our function field ingredients
  \cref{CH_thm} and
  \cref{bisep_bound_genus_char0}.

First, we handle the case $n = 2$ of a product of two elliptic curves. 
Let 
$E_1$ and $E_2$
be two non-isotrivial, non-isogenous elliptic curves over $K$, and let $\ell$ be a prime satisfying $\ell \geq c(g)$. Note that it follows that $\ell \geq c(0) = 17$. 
Take $H := \textnormal{Image}(\rho_{E_1 \times E_2,\ell})$ ,
$V_1 := E_1[\ell]$, and $V_2 = E_2[\ell]$.
We deduce from 
 \cref{CH_thm}  that $H$ projects surjectively onto $\SL_2(\Z/\ell\Z)$ in each factor of $\SL_2(\Z/\ell\Z)^2$.

We may then  apply \cref{MW_lemma} with the above $H, V_1$ and $V_2$ (the correct divisor $e \mid (\ell-1)$ is determined by this setup as described in \cite[p. 251]{MW_Galois}, while $B_1, B_2,$ and $D$ are as described in the statement of the referenced lemma). If $H = D$, then we are done. 
Otherwise, we have 
an isomorphism $f: V_1 \to V_2$ and a character $\chi_0$ on $G_K$, induced from the character $\chi$ on $H$ given by the lemma, such that, for every $\sigma \in G_K$, 
\[ \rho_{E_2,\ell}(\sigma) = \chi_0(\sigma)f \rho_{E_1,\ell}(\sigma) f^{-1}. \]
This means that $f$ induces a $G_K$-module isomorphism
$$
f_{\chi_0}: E_1[\ell] \longrightarrow E_{2, \chi_0}[\ell],
$$
where 
$E_{2, \chi_0}$ is the quadratic twist of $E_2$ by $\chi_0$;
so, 
for every $\sigma \in G_K$,
\[ \rho_{E_{2, \chi_0},\ell}(\sigma) = f_{\chi_0} \rho_{E_1,\ell}(\sigma) f_{\chi_0}^{-1}. \]

We note that
\[ \Gamma := \{(x,f_{\chi_0}x) \in E_1 \times E_{2, \chi_0} \mid x \in E_1[\ell]\} \]
is a subgroup of 
$E_1 \times E_{2, \chi_0}$ 
stable under the action of $G_{K}$. Indeed, if $\sigma \in G_{K}$ and $y := \rho_{E_1,\ell}(\sigma)x$, then 
\[ \sigma(x, f_{\chi_0}x) = (\rho_{E_1,\ell}(\sigma)x, \rho_{E_{2, \chi_0}, \ell}(\sigma)f_{\chi_0}x) = (y, f_{\chi_0}y) \in \Gamma. \]
Therefore, the abelian surfaces 
$E_1 \times E_{2, \chi_0}$ and $(E_1 \times E_{2, \chi_0})/\Gamma$ 
are both defined over $K$ 
and are isogenous via the natural quotient map 
$\pi: E_1 \times E_{2, \chi_0} \to (E_1 \times E_{2, \chi_0})/\Gamma$.

By our assumptions on 
$E_1, E_2$ and $k$,
\cref{bisep_bound_genus_char0}
provides the existence of an isogeny 
$$\psi_0: E_1 \times E_{2, \chi_0} \to (E_1 \times E_{2, \chi_0})/\Gamma$$
with 
$$\textnormal{deg}(\psi_0)  \leq  49^3 \cdot \mathcal{N}(g)^2 \cdot \max\{1,g^3\}.$$
We then have an endomorphism
\[ \epsilon := \widetilde{\psi_0} \circ \pi \in \textnormal{End}_{\overline{K}}(E_1 \times E_{2, \chi_0}) \simeq \mathbb{Z}^2 \]
of degree $\ell^2 \cdot \textnormal{deg}(\psi_0)^3$, which we can represent by a matrix $\begin{pmatrix} a & 0 \\ 0 & b \end{pmatrix}$ for some $a,b \in \mathbb{Z}$. 
As $\pi$ annihilates $\Gamma$, so too does $\epsilon$, and so for all $x \in E_1[\ell]$ we have
\[ ax = bf_{\chi_0}x = 0. \]
This means that the multiplication-by-$a$ map on $E_1$ and the multiplication-by-$b$ map on $E_{2, \chi_0}$ both annihilate 
the $\ell$-torsion (as $f_{\chi_0}$ is an isomorphism), and so $\ell \mid a$ and $\ell \mid b$. Moreover, 
$\ell^4 \mid a^2b^2 = \textnormal{deg}(\epsilon)$, 
so 
$\ell^2 \mid \textnormal{deg}(\psi_0)^3$. 
We can then rule out this case  by assuming that 
\[\ell >  40353607 \cdot \mathcal{N}(g)^{3} \cdot  \max\{1,g^{9/2}\}, \]
which we now do.
Thus, for all primes
\begin{align*}
\ell &>  40353607 \cdot \mathcal{N}(g)^{3} \cdot  \max\{1,g^{9/2}\} \\
&= \max\left\{c(g), 40353607 \cdot \mathcal{N}(g)^{3} \cdot  \max\{1,g^{9/2}\} \right\}, \\
\end{align*}
we have the desired equality $\Psi_\ell(E_1 \times E_2) = \SL_2(\Z/\ell\Z)^2$.


We  obtain the general case from the $n=2$ case, as follows. 
Let 
$E_1, E_2, \ldots, E_n$ 
be  non-isotrivial, pairwise non-isogenous elliptic curves over $K$, and let $\ell$ be a prime satisfying $\ell > \mathfrak{c}(g)$.
By the previous case of a product of two elliptic curves, the image of the map 
from $\Psi_\ell(E_1 \times E_2 \times \ldots \times E_n)$ to
$\SL(\Z/\ell\Z)^2$
via the projection of $\SL_2(\Z/\ell\Z)^n$ onto any two distinct factors is surjective. 
 As in the proof of \cite[Proposition 1]{MW_Galois}, we now use the result of Ribet \cite[Lemma 5.2.2]{Ribet_Galois} which implies that $\Psi_\ell(E_1 \times E_2 \times \ldots \times E_n) = \SL_2(\Z/n\Z)^n$ (noting that $\ell > c(0) = 17 > 5$). 
 This completes the proof of our main theorem.
\end{proof}

\begin{remark}
In our setup for \cref{effective_thm} (see also \cref{effective_theorem_introduction_version} stated in the introduction), 
we assume that $\textnormal{dim}(A) = n \geq 2$. The fact that $\mathfrak{c}(g) \geq c(g)$ for all $g$ 
allows the flexibility of assuming $n \geq 1$.
Of course, in the $n=1$ case, $\mathfrak{c}(g)$ can be a cruder bound, in which case one should defer to the bound $c(g)$ from \cref{CH_thm}. 
\end{remark}

\begin{remark}\label{non_uniform_remark}
We thank Ariel Weiss for pointing out to us that the character $\chi_0$ in the proof of \cref{effective_thm} yields an isogeny over $K$ with source  the abelian surface $E_1 \times E_{2, \chi_0}$. While Serre notes that the isomorphism $f$ yields an isomorphism between the $\ell$-torsion subgroups of $E_1$ and the twist $E_{2, \chi_0}$ \cite[p. 328]{Serre}, 
Masser--W{\"{u}}stholz trivialize $\chi_0$ to get an isogeny over the 
extension $K^{\chi_0}/K$, of degree at most $2$,
 instead of using the associated isogeny of abelian surfaces over $K$ as in our proof.
Using their isogeny estimates, this line of argument only costs 
Masser--W{\"{u}}stholz
 a factor of $2$ in their final bound compared to the argument that we use. 
 For us, the argument over $K$ importantly allows us to forego other possible arguments limiting the ramification of the characters that show up, which would be needed for the uniform version \cref{effective_thm}; the character $\chi_0$ depends on $\ell$, and it is not immediately clear that the extensions $K^{\chi_0}/K$ obtained as one ranges over $\ell$ have uniformly bounded genus.
\end{remark}

\appendix
\section{Function Field Frey--Mazur at Composite Level 
\\(with Benjamin Bakker)}\label{appendix}
\def\vol{\operatorname{vol}}
\def\mult{\operatorname{mult}}

\subsection{Introduction and strategy}\label{BT_thm_strategy}
This appendix is devoted to the extension of Bakker--Tsimerman's Frey--Mazur result over function fields over $\mathbb{C}$ \cite[Corollary 30]{BT16} to composite level, as we stated in \cref{Frey_Mazur_section} and used in \cref{deg_bound_section}. While an extension of their main results in terms of the gonality of the base field \cite[Theorem 2, Theorem 3]{BT16} should also be achievable via the methods used in their paper, we only require for our purposes a result in terms of the genus of the base field as in \cref{BT_thm}. We restrict to this result, which we recall below, for brevity of this appendix. 

\begin{theorem}[{cf. \cite[Theorem 2]{BT16}}]\label{BT_thm_appendix}
Let $K$ be the function field of a smooth, projective, and irreducible curve $C$ of genus $g$ over $\mathbb{C}$. Let $E_1$ and $E_2$ be non-isotrivial, non-isogenous elliptic curves over $K$. There exists a constant $\mathcal{N}(g)$, depending only on $g$, such that if there exists an $N$-congruence between $E_1$ and $E_2$, then $N \leq \mathcal{N}(g)$. 
\end{theorem}

We briefly recall here the strategy of \cite{BT16}, and  refer the reader to this referenced work for further background and details. The authors consider over $\mathbb{C}$ the surface $Z(N)$ parameterizing $N$-congruences of elliptic curves, which is realized as the quotient of $X(N) \times X(N)$
by the diagonal action of $\PGL_2(\Z/N\Z)$. 
Here, $X(N)$ is the modular curve parameterizing elliptic curves $E$ with a choice of projective isomorphism 
$E[N]\simeq (\Z/N\Z)^2$.
 An $N$-congruence of elliptic curves 
 $E_1[N] \simeq E_2[N]$ 
 over a function field $K$ over $\mathbb{C}$, with model $C$, 
 implies the existence of a map $C \to Z(N)$ over $\C$,
 and \cref{BT_thm_appendix} is then equivalent to the statement that any such curve $C$ of significantly large genus (or, in their case, gonality) must live on a Hecke divisor $H_M$ in $Z(N)$, 
 which parameterizes congruences between $M$-isogenous elliptic curves. 

Working in terms of the gonality $\gamma$ of the function field $K$, 
the authors take a gonal map $C \to \P^1$, 
which provides a map $\P^1 \to \Sym^\gamma(Z(N))$ to the symmetric product of $Z(N)$, 
and then lift this to a map
\[ C \to \left(X(N) \times X(N)\right)^\gamma. \]
They estimate the genus of $C$ in two ways in order to reach a contradiction for $N = p$ a significantly large prime, with the main work going into bounding the ramification of $C \to \P^1$ by showing that the ramification points of the map
\[\left( X(p) \times X(p)\right)^\gamma \to \Sym^\gamma(Z(p)) \]
have small incidence with $C$. This is done via the multiplicity estimates of \cite[Section 6]{BT16}, which in turn are fed by results that special points on curves in $X(p) \times X(p)$ tend to repel from \cite[Section 4]{BT16} and from general volume estimates for special subvarieties in hyperbolic manifolds proven in \cite[Section 5]{BT16}. 

\subsection{The requisite machinery}

In this section, which is the technical heart of this appendix, we generalize the necessary results from \cite{BT16} to prime-power levels. Once obtained, these will lead to a quick proof of \cref{BT_thm_appendix} along the same lines of the proof of the main theorem in the referenced work. In all that follows, we adopt the notation and conventions from \cite[\S 1.3]{BT16}, in particular regarding asymptotics and metrics, departing, as needed, from the notation of the previous sections of our paper. 
We emphasize that the letters $p$ and $k$ are now used for an arbitrary prime and an arbitrary positive integer, respectively,
while $g$ denotes either a group element or a non-negative integer.

We begin by borrowing terminology from \cite[\S 2.4]{BT16} for ramification points of the natural map $X(N)\times X(N) \to Z(N)$.
\begin{definition}
    Let $x,y \in X(N)$ be CM points induced by pairs $(E_x,\varphi_x)$, $(E_y,\varphi_y)$, and suppose that $x$ and $y$ have a common non-trivial stabilizer element $g \in \PGL_2(\Z/N\Z)$. Take lifts $g_x, g_y \in \GL_2(\Z/N\Z)$ of $g$ and automorphisms $h_x$, $h_y$ of $E_x$, $E_y$, respectively, having the same characteristic polynomial such that 
    \[ h_x \circ \varphi_x = \varphi_x \circ g_x^{-1}, \]
    and similarly with $x$ replaced by $y$. We call the point $(x,y) \in X(N)\times X(N)$ a \textbf{Heegner CM point} if the eigenvalues of the action of $h_x, h_y$ on the tangent spaces $T_0E_x$, $T_0E_y$ are the same, 
    and an \textbf{anti-Heegner CM point} otherwise. When the level $N$ is clear from context, we denote the sets of Heegner CM points and anti--Hegner CM points on $X(N)\times X(N)$ by $\CM^+, \CM^-$, respectively, and we set 
    \[ \CM := \CM^+ \cup \CM^-. \]
\end{definition}

\begin{definition}
    We call a point $(x,y) \in X(N)\times X(N)$ a \textbf{singular bicusp} if $x,y$ are cusps on $X(N)$ sharing a non-trivial stabilizer in $\PGL_2(\Z/N\Z)$. When the level $N$ is clear from context, we denote the set of singular bicusps on $X(N)\times X(N)$ by $\SBC$.
\end{definition}

\subsubsection{Injectivity radii and heights}
Recall that the \textbf{injectivity radius} $\rho_C(x)$ of a compact hyperbolic curve $C$ at a point $x\in C$ is half the minimum length of a geodesic loop based at $x$, and the \textbf{injectivity radius} $\rho_C$ of $C$ is the minimum 
\[ \rho_C := \min_{x\in C}\rho_C(x) . \]  
First and foremost, we have the following:
\begin{lemma}
    Equip  $X(N)$ with its hyperbolic metric of constant sectional curvature $-1$.  Then
    \[\rho_{X(N)}=2\log N +O(1)\]
\end{lemma}
\begin{proof}
    See \cite[Corollary 8]{BT16}.  The arguments therein do not require that $N$ is prime.
\end{proof}
Note also that the height estimates from \cite[\S3.3]{BT16} similarly generalize to $X(N)$; we leave this to the reader.

We view $X(N)$ as a Galois \'etale cover of $X(1)_N:=\Gamma(2,3,N)\backslash \HH$ with Galois group 
$G(N)\simeq \PGL_2(\Z/N\Z)$, 
where $\Gamma(2,3,N)$ is the triangle group with vertices $i,e^{i\theta_N},iy_N$.  The natural map $Y(1)\to X(1)_N$ gives a map on fundamental groups $\PSL_2(\Z)\to \Gamma(2,3,N)$, which induces an isomorphism $\PGL_2(\Z/N\Z)\to G(N)$ that we suppress from the notation.  Let $\iota,\varpi,\frak{y}\in X(N)$ be the images of $i,e^{i\theta_N},iy_N\in\HH$ 
and let $\bar H\subset G(N)$ be the stabilizer of $\frak{y}$.

\subsubsection{Repulsion of CM points and multiplicities}
\begin{lemma}[{cf. \cite[Proposition 14]{BT16}}]For any sufficiently small $\delta>0$, any prime $p$, and any sufficiently large 
positive integer $k$, we have the following. 
    Let $\xi,\xi'\in X(p^k)\times X(p^k)$ be distinct Heegner CM points points which have the same projections in $X(1)\times X(1)$.  If $d(\xi,\xi')<\delta\rho_{X(p^k)}$, then the projection of $\xi$ to $X(p^{k-O(k\delta)})\times X(p^{k-O(k\delta)})$ lies on the Hecke curve $T_d$ 
    for some
    $d=p^{O(k\delta)}$. 
\end{lemma}
\begin{proof}The proof in \cite[Proposition 14]{BT16} works with a very slight modification, so we include a sketch here.  
Setting $\xi :=( x, y)$ and $\xi' := (x', y')$, we may assume $\xi = (\iota,g ^{-1}\iota)$ (respectively, $(\varpi, g\varpi)$) for some $g \in G(p^k)$. Set 
$t : = 
\begin{pmatrix}
0 & 1 \\
-1 & 0 
\end{pmatrix}$ 
(respectively, 
$t :=
\begin{pmatrix}
0 & 1 \\
-1 & -1 
\end{pmatrix}$).  
Since $\xi$ is a Heegner CM point, $g$ commutes with $\bar t$, hence is in the $\Z/p^k\Z$-span of $\{1,\bar t\}$.  We then produce $M_x,M_y\in\SL_2(\Z)$ whose entries are of size $p^{O(k\delta)}$ with $x' := \bar M_xx$ and $y' := g^{-1}\bar M_y x$.  
Let $F\subset M_2(\Z)$ be the $\Z$-span of $\{1,t\}$, and note that it is a saturated submodule.  We have a map
\begin{equation}F\xrightarrow{f}M_2(\Z),\;\;f(X)=[t,M_y^{-1}XM_x]\label{commutator}\end{equation}
which is nonzero as in \cite[Lemma 15]{BT16}.  The element $g$ is in the kernel of $f$ mod $p^k$ and has unit determinant (so is in particular of additive order $p^k$).  
Let 
$B := \textnormal{Image} (f)$
and note that the obstruction to lifting $g$ to $\ker (f)$ is its image in $\operatorname{Tor}^\Z_1(\operatorname{coker} (f),\Z/p^k\Z)\subset B/p^kB$.  The  torsion of the cokernel has size $p^{O(k\delta)}$, as therefore does $\operatorname{Tor}^\Z_1(\operatorname{coker} (f),\Z/p^k\Z)$, so there is a (nonzero) element $D\in \ker (f)$ of size $p^{O(k\delta)}$ whose mod $p^k$ reduction up to scale differs from $g$ by an element which is annihilated by $p^{O(k\delta)}$, and in particular is a multiple of $p^{k-O(k\delta)}$.  This implies the claim.
\end{proof}

For the next statement on multiplicities of CM points on $X(p^k)\times X(p^k)$, we recall the following notation from \cite{BT16}: for a set $S$ of closed points of a curve $C$, we let 
\[ \mult_S(C) := \sum_{x \in S} \mult_x(C). \]
We also denote by $a(r)$ the area of a $1$-dimensional radius $r$ hyperbolic ball, so with our normalization $a(r)=4\pi\sinh^2(r/2)$. 

\begin{corollary}\label{multiplicity_CM}
For any sufficiently small $\delta>0$, any prime $p$, and any sufficiently large positive integer $k$, we have that for any non-Hecke curve $C$ in $X(p^k)\times X(p^k)$,
    \[\mult_{\mathrm{CM}^+}(C) =O(p^{-k\delta}\vol(C)).\]
\end{corollary}
\begin{proof}
    This is as in \cite[Proposition 25]{BT16}.  For sufficiently small $\delta_0$, we write $\mathrm{CM}^+=S\sqcup T$ where $T$ is the subset of Heegner CM points whose image under $\pi:X(p^k)\times X(p^k)\to X(p^{k(1-\delta_0)})\times X(p^{k(1-\delta_0)})$ lie on a Hecke curve of degree $\leq p^{k\delta_0}$.  By \cite[Theorem 19]{BT16}, for $\delta_1$ a sufficiently small multiple of $\delta_0$,
    \begin{align*}
        \mult_{\mathrm{CM}^+}(C) &\ll\frac{1}{a(\delta_1\rho_{X(p^k)})}\sum_{\xi\in S}\vol\left(C\cap B(\xi,\delta_1\rho_{X(p^k)})\right)+ \sum_{d\leq p^{k\delta_0}}(C.\pi^*T_d)\\
        &\ll\frac{\vol(C)}{a(\delta_1\rho_{X(p^k)})}+\sum_{d\leq p^{k\delta_0}}(T_d^*\pi_*C. \Delta)\\
        &\ll\frac{\vol(C)}{a(\delta_1\rho_{X(p^k)})}+\frac{\vol(C)}{a(\rho_{X(p^{k(1-\delta_0)})}/2)}\sum_{d\leq p^{k\delta_0}}\deg T_d\\
        &\ll \left(p^{-2k\delta_1}+p^{3k\delta_0}\cdot p^{-k(1-\delta_0)}\right)\vol(C).
    \end{align*}
    Choosing $\delta_0,\delta_1$ to be appropriate multiples of $\delta$, the claim is proven.
\end{proof}

\begin{remark}
    In \cite[Proposition 14]{BT16} and \cite[Proposition 25]{BT16}, Heegner and anti-Heegner CM points are treated together. In the preparation of this appendix, the authors of \cite{BT16} realized that this was in error; the argument in \cite[Proposition 14]{BT16} only holds for Heegner CM points, and one must provide a separate argument for anti-Heegner CM points as is done in \cite{BT13}, which treats the compact Shimura curve case. This will be handled in a forthcoming erratum \cite{BTerror} by Bakker--Tsimerman. 
\end{remark}

Modifying the argument of \cite{BTerror} in the same way, we obtain:
\begin{corollary}\label{multiplicity_anti_CM}
For any sufficiently small $\delta>0$, any prime $p$, and any sufficiently large positive integer $k$, we have that for any non-Hecke curve $C$ in $X(p^k)\times X(p^k)$,
    \[\mult_{\mathrm{CM}^-}(C) =O(p^{-k\delta}\vol(C)).\]
\end{corollary}

\subsubsection{Repulsion of bicusps and multiplicities}

\begin{lemma}[{cf. \cite[Proposition 17]{BT16}}]For any sufficiently small $\delta>0$, any prime $p$, and any sufficiently large positive integer $k$, we have the following.  If $\xi\in X(p^k)\times X(p^k)$ is within $(1/2+\delta)\rho_{X(p^k)}$ of at least three distinct singular bicusps, then all the singular bicusps within $(1/2+\delta)\rho_{X(p^k)}$ of $\xi$ project to the same Hecke curve $T_d$ in $X(p^{k-O(k\delta)})\times X(p^{k-O(k\delta)})$ 
for some 
 $d=p^{O(k\delta)}$. 
\end{lemma}
\begin{proof}
    Enumerate the set of singular bicusps $T=\{\zeta_j=(c_j,c_j')\}$ that are within $(1/2+\delta)\rho_{X(p^k)}$ of $\xi$.  The first paragraph of the proof in \cite{BT16} shows that for any two distinct $\zeta_j,\zeta_k$, both projections to $X(p^k)$ must be distinct, and that we may assume $\xi=(\iota,g\iota)$ for $g\in G(p^k)$.  For each $j$ we produce $M_{j},M_{j}'\in\SL_2(\Z)$ of size $p^{O(k\delta)}$ with $c=\bar M_j\frak{y}$ and $c'_j=g\bar M'_j\frak{y}$, and the argument of case (1) shows that the cosets $\bar M_jN(\bar H)$ are pairwise distinct, as are the cosets $\bar M_j'N(\bar H)$.  We have that $g\in\bigcap_{j}\bar M_jN(\bar H)\bar M_j'^{-1}\subset G(p^k)$.  Consider the map
 \begin{equation}M_2(\Z)\xrightarrow{f}\Z^T,\;\;f(X)_j=\mbox{lower left entry of }M_j^{-1}XM_j'\label{commutator}\end{equation}
    whose cokernel has size $p^{O(k\delta)}$.  Since $|T|\geq 3$, $\ker (f)$ is at most rank one given that any $X\in \ker (f)$ fixes at least 3 points on $\P^1(\R)$.  As before, $g$ is an element of the kernel mod $p^k$ of additive order $p^k$, so there exists an element $D\in \ker (f)$ of height $p^{O(k\delta)}$ whose reduction up to scale agrees with $g$ mod $p^{k-O(k\delta)}$.
    
\end{proof}

\begin{corollary}\label{multiplicity_cusps} 
For any sufficiently small $\delta>0$, any prime $p$, and any sufficiently large positive integer $k$, we have that for any non-Hecke curve $C$ in $X(p^k)\times X(p^k)$,
    \[p^k \cdot \mult_{\mathrm{SBC}}(C) =O(p^{-k\delta}\vol(C)).\]
\end{corollary}
\begin{proof}
    The proof is as in \cite[Proposition 26]{BT16} (with the correction from \cite{BTerror}) and the above corollary.  We leave the details to the reader.  Note that the proof of \cite[Proposition 12]{BT16} does not use that $p$ is prime.
    \end{proof}

\subsection{Proving \cref{BT_thm_appendix}}

Because we are working in terms of the genus of our function field rather than the gonality, we can side-step the arguments involving the symmetric product from the strategy outlined above and hence bypass having to generalize the results on multidiagonals from \cite{BT16}.  The argument therefore runs parallel to that of \cite{BT13}---see \cite[p. 711]{BT16} for related comments.

By the discussion in \cref{BT_thm_strategy}, \cref{BT_thm_appendix} is equivalent to the following statement, which we now prove. 

\begin{theorem}\label{BT_thm_appendix_restatement}
Let $g$ be a non-negative integer. There exists a positive integer $\mathcal{N}(g)$ such that, for all prime powers $p^k$, if $C \to Z(p^k)$ is a non-Hecke curve of genus $g$, then $p^k \leq \mathcal{N}(g)$. 
\end{theorem}

\begin{proof}
First, note that it suffices to prove the statement for prime powers $N = p^k$. Fix $g \in \mathbb{Z}^+$ and suppose to the contrary that for arbitrarily large prime powers $p^k$ we have a non-Hecke curve $C \to Z(p^{k})$ of genus $g$. For any sufficiently large prime power $p^k$, we take the normalization $C'$ of a component of the lift of $C$ with respect to the quotient map $q: X(p^k)\times X(p^k) \to Z(p^k)$. 
\[
\begin{tikzcd}
C' \arrow[r, "\psi"] \arrow[d, "\alpha"] & X(p^k)\times X(p^k) \arrow[d, "q"] \\
C \arrow[r]            & Z(p^k)            
\end{tikzcd}
\]
The ramification of $\alpha$ is supported on that of the map $q$, and hence occurs only at points in the fibers under $\psi$ over CM points and singular bicusps. The ramification index of $\psi^{-1}(\xi)$ for any $\xi \in \SBC$ is $O(p^k)$, while that for $\xi \in \CM$ is at most $3$. By \cref{multiplicity_CM}, \cref{multiplicity_anti_CM}, and \cref{multiplicity_cusps}, we then have
\[ \textnormal{Ram}(\alpha) \ll  |\psi^{-1}(\CM)| + p^k \cdot |\psi^{-1}(\SBC)| = o(\vol(C')). \]
Riemann--Hurwitz then gives
\begin{equation}\label{RH_B'_to_B}
g(C') = O(\deg(\alpha))+ o(\vol(C')) ,
\end{equation}
where the implied constant in the first term depends on $g$.  On the other hand, by the Schwarz lemma we have $g(C')\gg \vol(C')$, so
\[\vol(C')=O(\deg(\alpha)).\]
Letting $\deg_{X(p^k)}(C')$ be the maximum degree of the projections $C'\to X(p^k)$, we see that $\vol(C')$ is comparable to $\deg_{X(p^k)}(C') \cdot \vol(X(p^k))$, which is in turn comparable to $\deg_{X(p^k)}(C')\cdot |G(p^k)|$ (since $-\chi(X(1)_{p^k})=\frac{1}{6}-p^{-k}$).  Thus, $\deg_{X(p^k)}(C')$ is bounded.  From the commutative square
\[
\begin{tikzcd}
C' \arrow[r] \arrow[d, "\alpha"] & X(p^k) \arrow[d] \\
C \arrow[r]            & X(1)_{p^k}            
\end{tikzcd}
\]
we deduce $\deg_{X(1)_{p^k}}(C)$ is bounded.  As in \cite[\S 5]{BT13}, there is a bounded family of such curves $C\to X(1)_{p^k}\times X(1)_{p^k}$, and therefore finitely many possibilities for the image on fundamental groups (after removing the cusps).  For any non-Hecke curve, this image is Zariski dense in $\SL_2\times\SL_2$ by Andr\'e--Deligne \cite{andre}, and using \cite[Lemma 2.7]{Rapinchuk} in place of the theorem of Nori we have a contradiction.
\end{proof}


\bibliographystyle{alpha}
\bibliography{references}

\end{document}